\def\draftdate{April 5, 2012}

\documentclass{amsart}
\usepackage{amssymb}
\usepackage[v2,cmtip]{xy}

\usepackage{graphicx}

\let\term\emph

\def\botimes{\mathop{\textstyle \bigotimes}\nolimits}

\newcommand{\DRM}{\aD_{R^{\op}}}

\newcommand{\LC}[1]{\oC_{#1}}
\newcommand{\AC}{\oA}
\newcommand{\SK}{\oK}
\newcommand{\ie}{\mathbf{1}}
\newcommand{\ze}{\mathbf{i}}
\newcommand{\cim}[1]{#1_{\sharp}}

\newcommand{\MA}{R_{M}}

\newcommand{\opn}[1]{\Lambda_{#1}}

\newcommand{\cubea}{\mathbf{a}}
\newcommand{\cubeb}{\mathbf{b}}
\newcommand{\cubec}{\mathbf{c}}

\newcommand{\op}{{\mathrm{op}}}

\def\quickop#1{\expandafter\DeclareMathOperator\csname #1\endcsname{#1}}
\quickop{Id}\quickop{id}\quickop{Tor}\quickop{Ext}
\quickop{End}\quickop{Fun}\quickop{Hom}

\newcommand{\bL}{{\mathbb{L}}}
\newcommand{\bR}{{\mathbb{R}}}

\DeclareMathAlphabet{\scr}{U}{rsfs}{m}{n}

\let\catsymbfont\scr  
\newcommand{\aC}{{\catsymbfont{C}}}
\newcommand{\aD}{{\catsymbfont{D}}}
\newcommand{\aM}{{\catsymbfont{M}}}
\newcommand{\aS}{{\catsymbfont{S}}}
\newcommand{\aT}{{\catsymbfont{T}}}

\let\opsymbfont\mathfrak 
\newcommand{\oA}{{\opsymbfont{A}}}
\newcommand{\oC}{{\opsymbfont{C}}}
\newcommand{\oK}{{\opsymbfont{K}}}
\newcommand{\oL}{{\opsymbfont{L}}}

\newcommand{\iso}{\cong}
\newcommand{\sma}{\wedge}

\renewcommand{\to}{\mathchoice{\longrightarrow}{\rightarrow}{\rightarrow}{\rightarrow}}
\newcommand{\from}{\mathchoice{\longleftarrow}{\leftarrow}{\leftarrow}{\leftarrow}}

\newcommand{\subseg}[2]{%
\overbrace{\hbox to #2{\hss%
\vbox to 8pt{\vss\hrule height2pt width2pt depth1pt}%
\vbox to 8pt{\vss\hrule height1pt width#2 depth0pt}%
\vbox to 8pt{\vss\hrule height2pt width2pt depth1pt}%
\vbox to 0pt{\vss\hrule height0pt width0pt depth8pt}%
\hss}}^{#1}}
\newcommand{\subsegnolabel}[1]{%
\hbox to #1{\hss%
\vbox to 8pt{\vss\hrule height2pt width2pt depth1pt}%
\vbox to 8pt{\vss\hrule height1pt width#1 depth0pt}%
\vbox to 8pt{\vss\hrule height2pt width2pt depth1pt}%
\vbox to 0pt{\vss\hrule height0pt width0pt depth8pt}%
\hss}}

\newcommand{\subsegnoend}[2]{%
\overbrace{\hbox to #2{\hss%
\vbox to 8pt{\vss\hrule height2pt width2pt depth1pt}%
\vbox to 8pt{\vss\hrule height1pt width#2 depth0pt}%
\vbox to 8pt{\vss\hrule height1pt width2pt depth0pt}%
\vbox to 0pt{\vss\hrule height0pt width0pt depth8pt}%
\hss}}^{#1}}

\newcommand{\subsegnoel}[1]{%
\hbox to #1{\hss%
\vbox to 8pt{\vss\hrule height2pt width2pt depth1pt}%
\vbox to 8pt{\vss\hrule height1pt width#1 depth0pt}%
\vbox to 8pt{\vss\hrule height1pt width2pt depth0pt}%
\vbox to 0pt{\vss\hrule height0pt width0pt depth8pt}%
\hss}}

\newcommand{\subsegnobeg}[2]{%
\overbrace{\hbox to #2{\hss%
\vbox to 8pt{\vss\hrule height1pt width2pt depth0pt}%
\vbox to 8pt{\vss\hrule height1pt width#2 depth0pt}%
\vbox to 8pt{\vss\hrule height2pt width2pt depth1pt}%
\vbox to 0pt{\vss\hrule height0pt width0pt depth8pt}%
\hss}}^{#1}}

\newcommand{\subsegnobl}[1]{%
\hbox to #1{\hss%
\vbox to 8pt{\vss\hrule height1pt width2pt depth0pt}%
\vbox to 8pt{\vss\hrule height1pt width#1 depth0pt}%
\vbox to 8pt{\vss\hrule height2pt width2pt depth1pt}%
\vbox to 0pt{\vss\hrule height0pt width0pt depth8pt}%
\hss}}

\newtheorem{thm}[equation]{Theorem}
\newtheorem*{main}{Main Theorem}
\newtheorem*{cmain}{Conjecture}

\newtheorem{lem}[equation]{Lemma}
\newtheorem{prop}[equation]{Proposition}
\newtheorem{prob}[equation]{Problem}

\theoremstyle{definition}
\newtheorem{defn}[equation]{Definition}
\newtheorem{notn}[equation]{Notation}
\newtheorem{conv}[equation]{Convention}
\newtheorem{cons}[equation]{Construction}

\numberwithin{equation}{section}

\begin{document}

\title[The Smash Product for Derived Categories]
{The smash product for derived categories in stable homotopy theory}

\author{Michael A. Mandell}
\address{Department of Mathematics, Indiana University,
Bloomington, IN 47405}
\email{mmandell@indiana.edu}
\thanks{The author was supported in part by NSF grants DMS-0504069 and
DMS-0804272.}
\subjclass[2000]{Primary 55P43; Secondary 18D10, 18D50}
\date{\draftdate}

\begin{abstract}
An $E_{1}$ (or $A_{\infty}$) ring spectrum $R$ has a derived category of
modules $\aD_{R}$.  An $E_{2}$ structure on $R$ endows $\aD_{R}$ with
a monoidal product $\sma_{R}$.  An $E_{3}$ structure on $R$ endows
$\sma_{R}$ with a braiding.  If the $E_{3}$ structure extends to an
$E_{4}$ structure then the braided monoidal product $\sma_{R}$ is
symmetric monoidal.
\end{abstract}

\maketitle

\section*{Introduction}

Stable homotopy theory is essentially the study of generalized
homology and cohomology theories.  From its beginning in the work of
Spanier and Whitehead on duality in the 1950's and the work of Adams,
Atiyah and Hirzebruch, Thom, Quillen, and many others on vector
fields, topological $K$-theory, and cobordism theory in the 1950's,
1960's, and 1970's, stable homotopy theory has provided powerful tools
for studying questions in geometry and topology.  Many of algebraic
topology's deepest advances and greatest successes have been tied to
the development of new cohomology theories and the study of stable
phenomena.

Because cohomology theories involve long exact sequences, very few
algebraic constructions work without severe flatness hypotheses.
Stable homotopy theorists therefore study a refinement (due to
Boardman) of the category of cohomology theories, called the ``stable
category'', whose objects are usually called ``spectra''.  This
category has a ``smash product'' that captures multiplicative
structures on cohomology theories: Roughly speaking, multiplicative
cohomology theories tend to be represented by ``homotopical
ring spectra'', defined in terms of monoids for the smash product.
Actions of homotopical ring spectra define ``homotopical module
spectra'', which represent cohomology theories that are modules over 
ring theories.  Properties of homotopical ring spectra often extend
to simplify computations involving homotopical module spectra, and
vice-versa.

The stable category with its smash product provides a good context for
stable homotopy theory, and the notions of homotopical ring and module
spectra suffice for many purposes, as amply demonstrated in the
literature since the 1960's.  On the other hand, as 
addressed by May and collaborators by the mid 1970's and as became
widely acknowledged by the mid 1980's, certain necessary constructions
require a stronger point-set foundation. For example, homotopy ring
spectra are the stable analogue of homotopy associative
$H$-spaces rather than the analogue of topological monoids; because of
this, few of the constructions available in the stable category
preserve homotopical module spectra.

The papers \cite{ekmm,hss,mmss} rewrote the foundations of stable
homotopy theory, providing several categories whose homotopy
categories are the stable category but which have symmetric monoidal
point-set smash products (before passing to the homotopy category).
Current terminology calls the monoids and commutative monoids for
these smash products $S$-algebras and commutative $S$-algebras; these
are essentially equivalent to the older notions of $A_{\infty}$ and
$E_{\infty}$ ring spectra, respectively.
As a consequence of the modern foundations, 
for an $S$-algebra $R$, the category of point-set left (or right)
$R$-modules has an intrinsic homotopy theory.  The homotopy category,
usually called the ``derived category'', shares most of the 
structure of the stable category and admits most of the usual
constructions in homotopy theory, with the possible exception of those
that require an internal smash product.

In general, for an $S$-algebra $R$, we can form the balanced product
``$\sma_{R}$'' of a right $R$-module and a left $R$-module as a
functor from the derived categories to the stable category
\[
\sma_{R}\colon \DRM\times \aD_{R}\to \aS
\]
(where $\aD_{R}$ denotes the derived category of left
$R$-modules, $\DRM$ denotes the derived category of right
$R$-modules, and $\aS$ denotes the stable category).
As in the case of ordinary rings in algebra, when $R$ is a commutative
$S$-algebra, 
left and right $R$-modules are equivalent, and the balanced product
lifts to an internal smash product
\[
\sma_{R}\colon \aD_{R}\times \aD_{R}\to \aD_{R},
\]
which is a closed symmetric monoidal product.
Unlike the case of ordinary rings in algebra, ring spectra admit an
infinite hierarchy of structures between $S$-algebra and commutative
$S$-algebra, the $E_{n}$ hierarchy of Boardman and Vogt~\cite{BV}.
An $E_{1}$ ring spectrum is an $A_{\infty}$ ring spectrum,
is equivalent to an $S$-algebra, and has a derived category of left
modules.  An $E_{\infty}$ ring spectrum is equivalent to a commutative
$S$-algebra and its derived category has a symmetric monoidal
product.  This paper begins the study of the
derived categories of left modules over $E_{n}$ ring spectra for
$1<n<\infty$.   The main theorem is:

\begin{main}
Let $R$ be an $E_{2}$ ring spectrum.
\begin{enumerate}
\item The derived category of left modules $\aD_{R}$ is equivalent to
the derived category of right modules $\DRM$ and has a
closed monoidal product $\sma_{R}$ extending the balanced product.
\item If $R$ is an $E_{3}$ ring
spectrum, then $\sma_{R}$ has a braiding.
\item If $R$ is an $E_{4}$ ring spectrum then the braiding is a
symmetry, i.e., $\aD_{R}$ is a closed symmetric
monoidal category.
\end{enumerate}
\end{main}

As one of the principle interests in constructing $E_{\infty}$ structures on
ring spectra has been to have a monoidal or symmetric monoidal category of
modules, for statements in the derived category, now merely an $E_{2}$
or $E_{4}$ structure suffices.  For example, Maria Basterra and
the author have shown that the Brown Peterson spectrum $BP$ at each
prime is an $E_{4}$ ring spectrum \cite{BMBP}; it is currently not
known whether it is an $E_{\infty}$ ring spectrum.

To avoid a point of possible confusion, we emphasize that the derived
category $\aD_{R}$ in the theorem above is the derived category of
left modules for $R$ regarded as an $A_{\infty}$ ring spectrum, and not, for
example, the derived category of operadic modules for $R$ regarded as
an $E_{n}$ ring spectrum.  See Section~\ref{secoutline} for a review
of the precise definition of $\aD_{R}$.

The main theorem addresses only the question of derived categories or
homotopy categories.  In fact, the smash product in the homotopy
category derives from a point-set level ``lax monoidal product''
\cite[3.1.1]{leinsterbook} or ``partial lax monoidal product'', which we
outline in  Section~\ref{secrefine}. In
lectures on this work dating back to 2004, the author has presented
the following general 
conjecture, converse to the main theorem (in the $E_{2}$ case): 

\begin{cmain}
Under suitable technical hypotheses, a lax or partial lax monoidal
product on a category with structure maps weak equivalences induces an
$E_{2}$ structure on the derived endomorphism ring spectrum of the unit.
\end{cmain}

The previous conjecture generalizes the Deligne Hochschild
cohomology conjecture, which is the special case of the
monoidal category of $(A,A)$-bimodules over a ring (or DG ring or
$S$-algebra).  In this case, the derived endomorphism DG algebra (or ring
spectrum) is the (topological) Hochschild cohomology complex.  The
(affirmed) Deligne conjecture is that this is an $E_{2}$ algebra
\cite{McClureSmith}. 

More generally, the author has advertised the problem of identifying
the point-set structure on the category of modules over an $E_{n}$
ring spectrum (for $n>2)$, extending the lax monoidal structure.
Once identified, a corresponding converse conjecture could be
formulated.  With the new understanding of quasi-categories that
has developed in the time since the author first announced the main
theorem, the conjecture above and its generalization to $E_{n}$ ring
spectra (for all $n$) have become feasible to approach.  The author
understands that these and related problems have since been solved
by Clark Barwick \cite{HCA} and David Gepner \cite{Gepner}; see also
Lurie's treatment in \cite[2.3.15]{DAGVI}. 

\subsection*{Acknowledgments}
The author thanks Maria Basterra, Andrew Blumberg, and Tony Elmendorf
for useful conversations and suggestions.

\section{Outline and Preliminaries}\label{secoutline}

Although the constructions in this paper would presumably work in any modern
(topological) category of spectra, for definiteness we work in the
category of EKMM $S$-modules; this allows us to take some technical
shortcuts in several places using the fact that all
objects are fibrant.  For $E_{n}$ algebras, we work
exclusively with the \term{little 
$n$-cubes} operads $\LC{n}$ of Boardman and Vogt \cite{BV}: An
element of $\LC{n}(m)$ consists of $m$ almost disjoint
\term{sub-cubes} of the unit cube $[0,1]^{n}$,
labelled $1,\ldots,m$, of the form
\[
[x^{i}_{1},y^{i}_{1}]\times \cdots \times [x^{i}_{n}, y^{i}_{n}]
\]
(for $0\leq x^{i}_{j}<y^{i}_{j}\leq 1$, but generally not with equal
side lengths $y^{i}_{j}-x^{i}_{j}$: These are affinely embedded sub-cubes,
rather than actual geometric sub-cubes). 
An $E_{n}$ algebra
in this 
context is then an $S$-module $R$ together with an action
\[
\LC{n}(m)_{+}\sma_{\Sigma_{m}} R^{(m)} \to R
\]
satisfying the usual properties (where $R^{(m)}=R \sma_{S} \cdots
\sma_{S} R$).  As a technical remark for those familiar with $E_{n}$
ring spectra in the sense of Lewis and May \cite{lms}, we note
that this is precisely an $E_{n}$ ring spectrum $R$ for the operad
$\LC{n}\times \oL$ (where $\oL$ denotes the linear isometry operad)
such that the underlying $\bL$-spectrum of $R$
(q.v.~\cite[II.4.1--2]{ekmm}) is an 
$S$-module \cite[II.1.1]{ekmm}.  The usual theory \cite{MayGILS}
(cf.~\cite[XII\S1,II\S4]{ekmm}) shows that any other sort of $E_{n}$
ring spectrum is equivalent to one of this type in an essentially
unique way.

We denote by $\AC$ the non-$\Sigma$ operad of little $1$-cubes: An element
of $\AC(k)$ is a sequence of $k$ almost disjoint
sub-intervals of the unit interval in order.
Then $\AC(k)\subset \LC{1}(k)$, and as an operad $\LC{1}\iso \AC \times
\Sigma$; thus, $\AC$-algebras and 
$\LC{1}$-algebras coincide.  We regard $\LC{n}$-algebras as
$\AC$-algebras via the usual inclusion of $\LC{1}$ in $\LC{n}$ (taking
a sub-interval $[x,y]$ to the sub-cube $[x,y]\times [0,1]^{n-1}$).

For a $\LC{n}$-algebra $R$, we understand a left $R$-module to be an
operadic left module for $R$ regarded as an $\AC$-algebra.  In other
words, a left $R$-module consists of an $S$-module $M$ and maps of
$S$-modules 
\[
\AC(m+1)_{+}\sma R^{(m)}\sma_{S}M\to M
\]
for all $m$, satisfying the usual associativity and unit diagrams (as
in, for example, \cite[I.4.2.(ii)]{kmbook}), reviewed in
Section~\ref{seclmea}. We use $\aM_{R}$ to denote the category of left
$R$-modules.  For purely formal reasons, $\aM_{R}$ is a category of
modules over an $S$-algebra $U_{\AC}R$ (or just $UR$), the \term{left
module enveloping algebra} of $R$, which we review in
Section~\ref{seclmea}.  In fact, using the details of the little
$1$-cubes non-$\Sigma$ operad $\AC$, we give a concrete description of
$UR$.  Using that description, we prove the following result on
enveloping algebras.  This result is a special feature of $\AC$ not
shared by a general $A_{\infty}$ operad without additional hypotheses
on the $A_{\infty}$ algebra $R$.

\begin{thm}\label{thmACnice}
For any $\AC$-algebra $R$, the canonical map of left $UR$-modules $UR\to R$
induced by the unit of $S\to R$ is a homotopy equivalence of $S$-modules.
\end{thm}

We understand the derived category of left $R$-modules $\aD_{R}$ to be
the derived category of $UR$-modules $\aD_{UR}$ \cite[III\S2]{ekmm},
obtained by formally inverting the weak equivalences.  We build
the smash product on $\aD_{R}$ in the Main Theorem by combining a
formal construction on $\aM_{R}$ with some homotopical results.

The formal construction involves the ``interchange'' property of the
operads $\LC{1}$ and $\LC{n-1}$ for a $\LC{n}$-algebra $R$.  Pairwise
cartesian product of sub-cubes defines a map
\[
\LC{1}(\ell)\times \LC{n-1}(m)\to \LC{n}(\ell m)
\]
that is a \term{pairing of operads} \cite{MayPairings}.
We use this pairing in Section~\ref{secpair} to associate to every
element of $\LC{n-1}(m)$ a natural map
of $\AC$-algebras 
\[
R\sma_{S}\cdots \sma_{S}R\to R
\]
and hence a map of $S$-algebras
$U(R^{(m)})\to UR$.  As a variant of 
this, for any space $X$  and map $f\colon X\to \LC{n-1}(m)$, 
$UR\sma X_{+}$ becomes a $(UR,U(R^{(m)}))$-bimodule, and hence defines a
functor 
\[ 
f_{*}\colon \aM_{R^{(m)}}\to \aM_{R}, \qquad f_{*}M = UR\sma X_{+}\sma_{U(R^{(m)})}M.
\]
The diagonal map $\AC\to \AC^{m}$ defines a map of $S$-algebras
\begin{equation}\label{eqURdiag}
U(R^{(m)})=U_{\AC}(R^{(m)})\to U_{\AC^{m}}(R^{(m)})\iso
(UR)^{(m)},
\end{equation}
which defines a forgetful or pullback functor
\[
\aM_{(UR)^{(m)}}\to \aM_{U(R^{(m)})}=\aM_{R^{(m)}}.
\]
Composing these functors with the smash product over $S$
\[
\aM_{R}\times \cdots \times \aM_{R}=
\aM_{UR}\times \cdots \times \aM_{UR}\to
\aM_{(UR)^{(m)}},
\]
we obtain a functor 
\begin{equation}\label{eqdefopn}
\opn{f}\colon \aM_{R}\times \cdots \times \aM_{R}\to \aM_{R}.
\end{equation}
In other words,
\[
\Lambda_{f}(M_{1},\ldots,M_{m})=UR\sma X_{+}\sma_{U(R^{(m)})}(M_{1}\sma_{S}\cdots \sma_{S}M_{m}).
\]
We call these operations \term{$E_{n}$ interchange operations}.

When $n=2$, we use $X=*$ and $f$ the element $\mu =([0,1/2],[1/2,1])$
\[
\subseg{1/2}{4em}\subseg{1/2}{4em}
\]
of $\AC(2)\subset \LC{1}(2)$ to construct a functor $\opn{\mu}$ that
provides point-set version of
the smash product functor for the Main Theorem.  We use $X$ an
interval and $f$ a path $\alpha$ from
\begin{gather*}
\subseg{1/4}{2em}\subseg{1/4}{2em}\subseg{1/2}{4em}
\qquad \text{to}\qquad
\subseg{1/2}{4em}\subseg{1/4}{2em}\subseg{1/4}{2em},
\end{gather*}
in $\AC(3)\subset \LC{1}(3)$ as a key component of the construction
of the associativity 
isomorphisms (in $\aD_{R}$) for the smash product (see
also~\ref{eqconsass} below).  We use maps 
from the pentagonal disk to $\LC{1}$ to establish coherence; see
Section~\ref{secsma} for details.  For $n=3$, we use a path like the
one pictured
\[
\vcenter{\hbox{\includegraphics[width=.15\hsize]{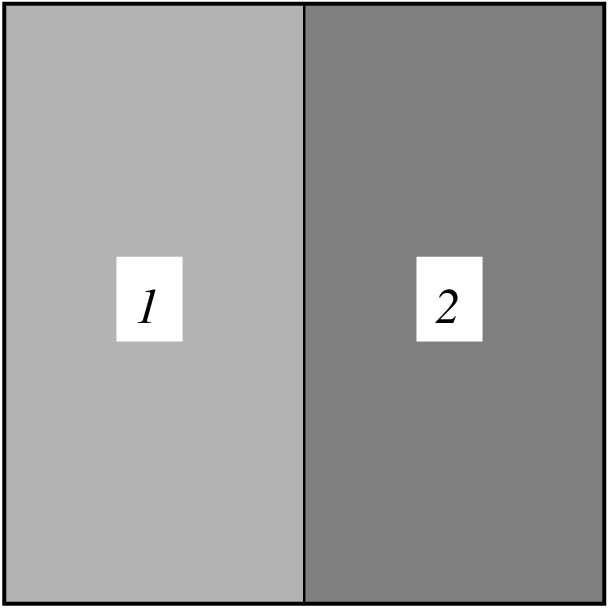}}}
\qquad
\vcenter{\hbox{\includegraphics[width=.15\hsize]{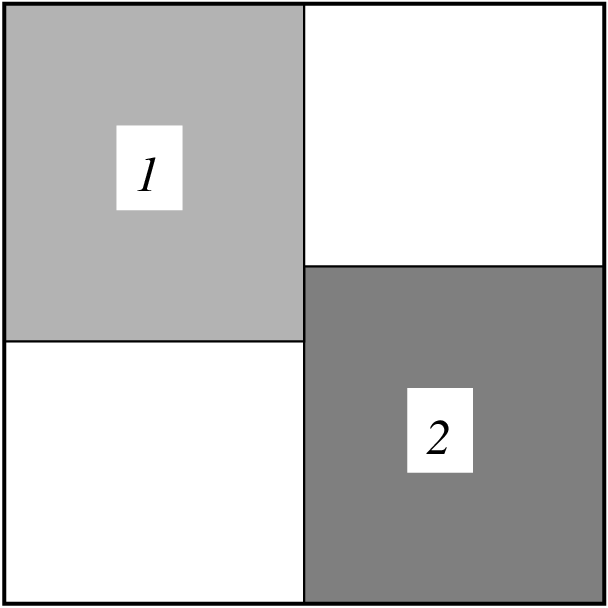}}}
\qquad
\vcenter{\hbox{\includegraphics[width=.15\hsize]{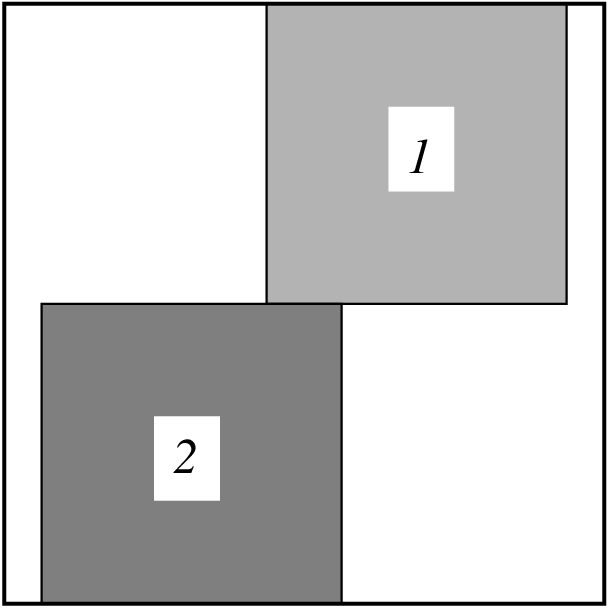}}}
\qquad
\vcenter{\hbox{\includegraphics[width=.15\hsize]{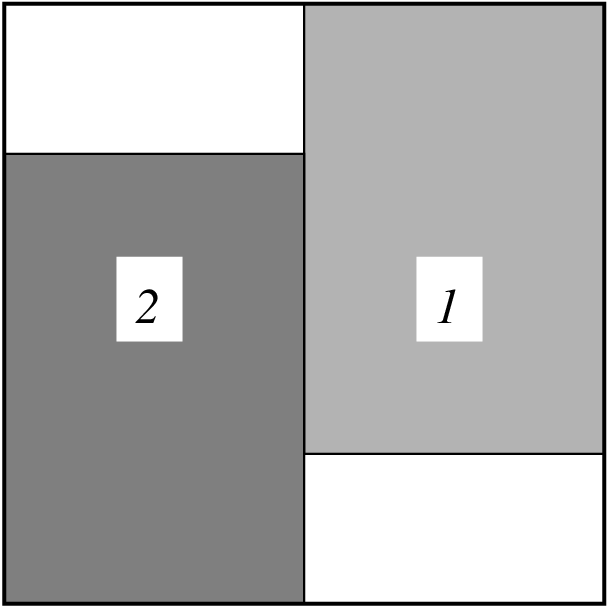}}}
\qquad
\vcenter{\hbox{\includegraphics[width=.15\hsize]{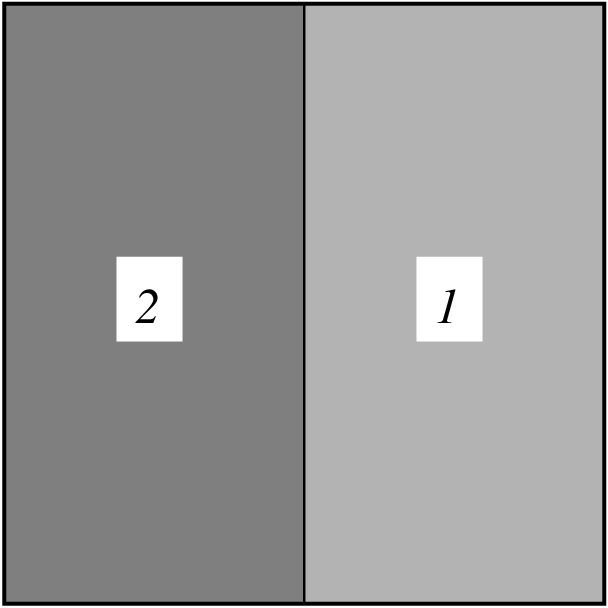}}}
\]
in $\LC{2}(2)$ to construct a braiding, and for $n=4$, a null homotopy
in $\LC{3}(2)$ of the composition of such paths to prove the symmetry.  See
Section~\ref{secsma} for details.

The $E_{n}$ interchange operations $\opn{f}$ do not strictly preserve 
composition, and this introduces some complications into the formal
picture.  To illustrate, let $f\colon X\to \LC{1}(2)$ and $g\colon Y\to
\LC{1}(2)$.  We obtain a map $f\circ_{2}g\colon X\times Y\to \LC{1}(3)$
by operadic composition and hence a functor
\[
\opn{f\circ_{2}g}\colon \aM_{R}\times \aM_{R}\times \aM_{R}\to \aM_{R},
\]
which is defined by
\[
\opn{f\circ_{2}g}(L,M,N)=UR\sma (X\times Y)_{+}\sma_{U(R^{(3)})}
(L\sma_{S}M \sma_{S}N).
\]
On the other hand, the composition of operations
$\opn{f}\circ_{2}\opn{g}$ is the functor
\begin{multline*}
\opn{f}(L,\opn{g}(M,N))=
UR\sma X_{+}\sma_{U(R^{2})}
  (L\sma_{S}(UR\sma Y_{+}\sma_{U(R^{(2)})}(M\sma_{S}N)))\\
\iso UR \sma (X\times Y)_{+}\sma_{UR\sma_{S}U(R^{(2)})}
  (L\sma_{S}M\sma_{S}N).
\end{multline*}
Specifically, $\opn{f\circ_{2}g}$ treats the $(UR)^{(3)}$-module
$L\sma_{S}M\sma_{S}N$ as a $U(R^{(3)})$-module, while
$\opn{f}\circ_{2}\opn{g}$ treats it as a
$UR\sma_{S}U(R^{(2)})$-module.  A generalization of~\eqref{eqURdiag}
induces a map 
of $S$-algebras from 
$U(R^{(3)})$ to $UR\sma_{S}U(R^{(2)})$, and so induces a natural
transformation 
\[
\opn{f\circ_{2}g}\to \opn{f}\circ_{2}\opn{g}.
\]
More generally, for a $\LC{n}$-algebra $R$,
given a map $f\colon X\to
\LC{n-1}(m)$ and maps $g_{i}\colon Y_{i}\to \LC{n-1}(j_{i})$, we have
a natural transformation 
\begin{equation}\label{eqopmap}
\opn{f\circ (g_{1},\ldots,g_{m})}
\to \opn{f}\circ (\opn{g_{1}},\ldots,\opn{g_{m}})
\end{equation}
of functors $\aM_{R}\times \cdots \times \aM_{R}$ to $\aM_{R}$.
Although this transformation is not an isomorphism, in
Section~\ref{secmc}, we show that it 
is often a weak equivalence.

\begin{thm}\label{thmopcomp}
With notation as above, for cofibrant $R$-modules $M_{1},\ldots,
M_{j}$ with $j=j_{1}+\cdots
+j_{m}$, the natural map~\eqref{eqopmap} 
\[
\opn{f\circ (g_{1},\ldots,g_{m})}(M_{1},\ldots,M_{j})
\to \opn{f}(\opn{g_{1}}(M_{1},\ldots,M_{j_{1}}),\ldots,
\opn{g_{m}}(M_{j-j_{m}+1},\ldots,M_{j}))
\]
is a weak equivalence.
\end{thm}

We apply Theorem~\ref{thmopcomp} in Section~\ref{secsma}
to construct the coherence isomorphisms in
$\aD_{R}$ for the Main Theorem.  For example, for $\mu\in \LC{1}(2)$
and $\alpha\colon I\to \LC{1}(3)$ as above, the maps
\begin{equation}\label{eqconsass}
\opn{\mu}\circ_{2}\opn{\mu}
\from \opn{\mu\circ_{2}\mu}\to \opn{\alpha}
\from \opn{\mu\circ_{1}\mu}\to \opn{\mu}\circ_{1}\opn{\mu}
\end{equation}
induce isomorphisms in $\aD_{R}$, which construct the associativity
isomorphism for the Main Theorem.  See Section~\ref{secsma} for
details.  To make this work and to use~\eqref{eqconsass} to construct
an isomorphism of left derived functors, we need to understand
composition of the left derived functors of the operations $\opn{f}$.
For this, we 
have the following theorem proved in Section~\ref{secmc}.

\begin{thm}\label{thmhocof}
Let $R$ be a $\LC{n}$-algebra, $f\colon X\to
\LC{n-1}(m)$ a map, and $M_{1},\ldots,M_{m}$ $R$-modules.  If $X$ is
homotopy equivalent to a CW complex and $M_{1},\ldots, M_{n}$ are
homotopy equivalent to cofibrant $R$-modules, then
$\opn{f}(M_{1},\ldots,M_{m})$ is homotopy equivalent to a cofibrant
$R$-module. 
\end{thm}

This theorem in particular implies that the left derived functor of a
composite of $E_{n}$ interchange operations is the corresponding 
composite of derived functors. 

\subsection*{Outline}
In Section~\ref{seclmea}, we review the left module enveloping algebra
and prove Theorem~\ref{thmACnice}.  In Section~\ref{secmc}, we study
the homotopy theory of $R$-modules and the operations $\opn{f}$; we
prove Theorems~\ref{thmopcomp} and~\ref{thmhocof}.  In
Section~\ref{secsma}, we apply this theory to prove the Main Theorem.
Section~\ref{secrefine} discusses the point-set lax monoidal
refinement of the constructions that go into the proof of the Main
Theorem.  Section~\ref{secrefine} also discusses the converse
conjecture in the introduction and further generalizations of the
Deligne conjecture (and their converses).

The final section, Section~\ref{secmoore}, bears no direct
relationship to the Main Theorem, but rather provides a follow-up to the
proof of Theorem~\ref{thmACnice} and the concrete description of the
left module enveloping algebra $UR$.  For an $\AC$-algebra $R$, an
alternative concrete construction, like the construction of the Moore
loop space, produces an associative algebra $\MA$ that we call the ``Moore
algebra''.  In Section~\ref{secmoore}, we construct a natural zigzag
of weak equivalences between the left module enveloping algebra $UR$
and the Moore algebra $\MA$.  This then relates the categories of
$R$-modules to $\MA$-modules. 

\section{The Left Module Enveloping Algebra}\label{seclmea}

For an $\AC$-algebra $R$, a left $R$-module
consists of an $S$-module $M$ together with action maps
\[
\xi_{m}\colon \AC(m+1)_{+}\sma R^{(m)}\sma M\to M
\]
satisfying the usual conditions.  Writing $\zeta$ for the $\AC$-algebra
multiplication of $R$, these conditions are the
associativity diagrams
\[ \def\objectstyle{\scriptstyle}\def\labelstyle{\scriptstyle}
\xymatrix@R-.5pc@C-1pc{%
\bigl(\AC(m+1)\times \bigl(\AC(j_{1})\times \cdots \times
\AC(j_{m})\times \AC(j_{m+1}+1)\bigr)\bigr)_{+} \sma R^{(j)}\sma_{S}M
\ar[r]^(.68){\relax\circ \sma\id}\ar[d]_{\iso}
&\AC(j+1)_{+}\sma R^{(j)}\sma_{S}M
\ar[dd]^{\xi_{j}}
\\
\AC(m+1)_{+}\sma 
\bigl((\AC(j_{1})_{+}\sma R^{(j_{1})}) \sma_{S}\cdots 
\sma_{S}(\AC(j_{m})_{+}\sma R^{(j_{m})})
\sma_{S} (\AC(j_{m+1}+1)_{+}\sma R^{(j_{m})}\sma_{S}M)\bigr){\hbox to -3.5pc{\hss}}
\ar[d]_{\id\sma \zeta_{j_{1}}\sma\cdots \sma\zeta_{j_{m}}\sma\xi_{j_{m+1}}}\\
\AC(m+1)_{+}\sma R^{(m)}\sma_{S}M\ar[r]_{\xi_{m}}
&M
}
\]
(for $m,j_{1},\ldots,j_{m+1}\geq 0$ and $j=j_{1}+\cdots +j_{m+1}$)
and the unit diagram
\[
\xymatrix{%
\{\ie\}_{+}\sma R^{(0)}\sma_{S}M\ar[r]\ar[d]_{\iso}&\AC(1)_{+}\sma
R^{(0)}\sma_{S}M \ar[d]^{\xi_{0}}\\
S^{0}\sma S\sma_{S}M\ar[r]_{\iso}&M,
}
\]
where $\ie$ denotes the identity element of $\AC(1)$ (the whole
sub-interval $[0,1]$ of $[0,1]$).
The action maps (as generators), the associativity diagrams (as
relations), and the unit diagram (as the unit) implicitly specify an
$S$-algebra $UR$ that encodes an $R$-module structure.  We begin with
this construction.

Let $UR$ be the $S$-module formed as the coequalizer of the following diagram:
\[  \def\objectstyle{\displaystyle}
\xymatrix@C-1pc{%
\bigvee_{m,j_{1},\ldots,j_{m}}
\bigl(\AC(m+1)\times \bigl(\AC(j_{1})\times \cdots \times \AC(j_{m})\bigr)\bigr)_{+}\sma R^{(j)}
\ar[r]<2ex>\ar[r]<.5ex>&
\bigvee_{m}\AC(m+1)_{+}\sma R^{(m)},
}
\]
where one map is induced by the operadic multiplication $\circ$ and
the other by the $\AC$-algebra multiplication of $R$.  The operadic
composition of $\AC$ using the last sub-interval,
\[
\AC(m+1)\circ_{m+1}\AC(k+1) \to \AC(m+k+1),
\]
induces a multiplication map $UR\sma_{S}UR\to UR$, and the inclusion of the
element $\ie$ in $\AC(1)$ induces a unit map $S\to UR$.  Since the
operadic composition is associative and unital,
\[
\cubea\circ_{\ell+1} (\cubeb \circ_{m+1} \cubec) = 
(\cubea \circ_{\ell+1} \cubeb)\circ_{\ell+m+1}\cubec
\qquad \text{and} \qquad
\ie \circ_{1} \cubea = \cubea = \cubea\circ_{m+1} \ie,
\]
it follows that the multiplication and unit maps make $UR$ into an
associative $S$-algebra.

\begin{defn}\label{deflmenv}
The $S$-algebra $UR$ is called the \term{left module enveloping algebra}.
\end{defn}

Comparing the universal property defining $UR$ with the data defining
a left $R$-module leads to the following proposition
(cf. \cite[1.6.6]{gk}, \cite[I.4.10]{kmbook}):

\begin{prop}\label{propenvalg}
A left $R$-module structure on an $S$-module determines and is
determined by a left $UR$-module structure.
\end{prop}

\begin{conv}\label{convUR}
By slight abuse, we use left $R$-modules and left $UR$-modules
interchangeably. We define the category of left $R$-modules $\aM_{R}$
to be the category of left $UR$-modules $\aM_{UR}$.
\end{conv}

The construction of $UR$ above is purely formal, using none of the
specifics of $\AC$; indeed, the analogue of construction makes sense
for an arbitrary non-$\Sigma$ operad, and Proposition~\ref{propenvalg}
holds in full generality.  On the other hand, for the non-$\Sigma$
operad $\AC$, the left module enveloping algebra admits a more
concrete description, which we now produce.

Let $D$ denote the subspace of $\AC(2)$ where the first sub-interval
begins at zero and the second sub-interval begins at the same point
where the first one ends:
\[
D=\left\{ ([x_{1},y_{1}],[x_{2},y_{2}])\in \AC(2) \mid x_{1}=0,
y_{1}=x_{2} \right\}.
\]
Let $\bar D=\AC(1)$; then dropping the
first sub-interval includes $D$ in $\bar D$ as the subspace of
intervals that do not start at $0$.  Let $A=AR$ be the $S$-module defined
by the following pushout diagram.
\[
\xymatrix@-1pc{%
D_{+}\sma S\ar[r]\ar[d]&D_{+}\sma R\ar[d]\\
\bar D_{+}\sma S\ar[r]&A
}
\]
Intuitively, $A$ consists of pairs of sub-intervals $([0,a],[a,b])$, with the
first sub-interval labelled by $R$, union sub-intervals $[0,b]$
labelled by $S$. 

We use $\circ_{2}$ to construct an associative multiplication on $A$
as follows. Given $\cubea =([0,a],[a,b])$ and $\cubec
=([0,c],[c,d])$ in $D$, then $\cubea
\circ_{2}\cubec$ ``plugs'' $\cubec$ into the second sub-interval in
$\cubea$, producing three sub-intervals,
\[
\underbrace{%
\subseg{a}{4em}%
\subseg{(b-a)c}{8em}%
\subseg{(b-a)(d-c)}{6em}%
\subsegnoel{6em}}_{b}
\subsegnobl{2em}\ .
\]
These define a new element of $D$ by taking the first sub-interval to
be the concatenation of the first two sub-intervals above and taking the
second sub-interval to be the remaining (third) sub-interval
above.  In formulas, this is the pair
\[
([0,a+(b-a)c],[a+(b-a)c,a+(b-a)d]),
\]
and pictorially is
\[
\subseg{a+(b-a)c}{12em}%
\subseg{(b-a)(d-c)}{6em}
\subsegnolabel{8em}\ .
\]
This defines a map $p\colon D\times D\to D$.

To explain what happens with the $R\sma_{S}R$ factor of $(D_{+}\sma
R)^{(2)}$, we use the first two sub-intervals in $\cubea
\circ_{2}\cubec$ to specify a map 
$q\colon D\times D\to \AC(2)$.  Let $q$ be the map sending
$(\cubea,\cubec)$ as above to 
\[
([0,a/(a+(b-a)c)],[a/(a+(b-a)c),1]),
\]
the pair obtained by taking just the first two sub-intervals of $\cubea
\circ_{2}\cubec$ and rescaling to length $1$,
\[
\underbrace{%
\subseg{a}{4em}%
\subseg{(b-a)c}{8em}}_{a+(b-a)c}\ .
\]
In other words, $p$ and $q$ decompose the composition
$\circ_{2}$ into two steps, 
\[
\cubea\circ_{2}\cubec =p(\cubea,\cubec)\circ_{1}q(\cubea,\cubec).
\]

Using $p\times q\colon D\times D\to D\times \AC(2)$ and the $\AC$-algebra
structure of $R$, we get a map
\[
(D_{+}\sma R)\sma_{S} (D_{+}\sma R)\iso (D\times D)_{+}\sma R\sma_{S} R\to
D_{+}\sma \AC(2)_{+}\sma R\sma_{S} R\to D_{+}\sma R.
\]
(In words, we multiply the
$D$ factors using $p$ and the $R$ factors according to $q$.)
Easy computations show that this extends to a map $A\sma_{S} A\to A$
that is associative and unital (with unit $S\to A$ induced by $\ie\in
\bar D$), 
making $A$ an associative $S$-algebra.  The construction $A=AR$ is
clearly functorial in $\AC$-algebra maps of $R$, and we get the
following proposition. 

\begin{prop}\label{propconsUR}
The construction $A$ above defines a functor from $\AC$-algebras to
associative $S$-algebras.
\end{prop}

The inclusion of $D$ in $\AC(2)$ and the inclusion of $\bar D$ in
$\AC(1)$ induce a natural map $\phi$ of $S$-modules under $S$ from $AR$
to $UR$. 

\begin{thm}\label{thmUR}
The map $\phi \colon AR\to UR$ is a natural isomorphism of associative
$S$-algebras. 
\end{thm}

\begin{proof}
For $m>0$, let $f_{m}\colon \AC(m+1)\to D\times
\AC(m)$ denote the map that sends 
$([x_{1},y_{1}],\ldots,[x_{m+1},y_{m+1}])$ of $\AC(m+1)$ to 
\[
([0,x_{m+1}],[x_{m+1},y_{m+1}]),
([x_{1}/x_{m+1},y_{1}/x_{m+1}],\ldots,[x_{m}/x_{m+1},y_{m}/x_{m+1}]),
\]
and let $f_{0}$ be the identity map $\AC(1)=\bar D$.
Then for every $m>0$, $j=j_{1}+\cdots+j_{m}>0$, the following diagram commutes,
\[
\xymatrix{%
\AC(m+1)\times \AC(j_{1})\times \cdots \times \AC(j_{m})
\ar[r]^(.7){\circ}\ar[d]_{f_{m}\times \id}
&\AC(j+1)\ar[d]^{f_{j}}\\
D\times \AC(m)\times \AC(j_{1})\times \cdots \times \AC(j_{m})
\ar[r]_(.7){\id\times \circ}
&D\times \AC(j)
}
\]
and the analogous diagram for $j=0$ commutes.
It follows that the composite 
\[
\AC(m+1)_{+}\sma R^{(m)}\to D_{+}\sma \AC(m)_{+}\sma R^{(m)}\to
D_{+}\sma R
\]
for $m>0$ and the identity map $\AC(1)_{+}\sma S=\bar D_{+}\sma S$
induce a map $\epsilon \colon UR\to A$, and it is easy to see from
the definition of the $S$-algebra structures that $\epsilon$ is a
map of associative $S$-algebras.  The composite $\epsilon \circ \phi$
is the identity on $A$.  Since the composite of $f_{m}$ with
$\circ_{1}$ is the identity on 
$\AC(m+1)$ for $m>0$, the defining map
\[
\bigvee_{m}\AC(m+1)_{+}\sma R^{(m)} \to UR
\]
factors through $\phi\circ \epsilon$, and so $\phi \circ \epsilon$ is
the identity on $UR$. 
\end{proof}

We close this section with the proof of Theorem~\ref{thmACnice}: We
show that the canonical map of $UR$-modules $UR\to R$
induced by the inclusion of the unit $S\to R$ is a homotopy
equivalence of $S$-modules.  In terms of the model $A$ above, we can
identify this as the map $\chi \colon A\to R$ induced by the map
\[
D_{+}\sma R\to \AC(1)_{+}\sma R\to R
\]
that forgets the \emph{second} sub-interval in $D$ and applies the
$\AC$-algebra multiplication map $\AC(1)_{+}\sma R\to R$.  We obtain a
map back  
$\psi \colon R\to A$ as the composite
\[
R\to D_{+}\sma R \to A
\]
induced by the inclusion of $([0,1/2],[1/2,1])$ in $D$.  (This map of
$S$-modules is clearly not a map of $UR$-modules.) We have a
homotopy $H_{t}$ from $\psi \circ \chi$ to the identity on $A$ induced
by the linear homotopy
\[
H_{t}([0,c],[c,d])=([0,1/2+t(c-1/2)],[1/2+t(c-1/2),(1-t)+td])
\]
on $D$ (and $\bar D$); note that $(1-t)+td > 1/2+t(c-1/2)$ since
$(1-t)/2+t(d-c)>0$.  On the other side, we have a homotopy $G_{t}$
from $\chi \circ \psi$ to the identity on $R$ induced by the path
$G_{t}=([0,1/2+t/2])$ in $\AC(1)$ and the $\AC$-algebra
multiplication.  This completes the proof of Theorem~\ref{thmACnice}.

\section{The Interchange Operations}
\label{secmc}\label{secpair}

This section constructs the interchange operations and studies them
from the perspective of the homotopy theory of $R$-modules.
Specifically, we prove Theorems~\ref{thmopcomp} and~\ref{thmhocof},
which let us understand the left derived functors and their
compositions. We begin with the point set construction.
Throughout this section $n$ and the $\LC{n}$-algebra $R$ remain fixed,
but we note that all constructions are functorial in the
$\LC{n}$-algebra $R$ and in the inclusions $\LC{n}\to \LC{n'}$ for
$n'>n$. 

For $\ell,m\geq 0$, let $\rho \colon \AC(\ell)\times \LC{n-1}(m)\to
\LC{n}(\ell m)$ be the map that takes the pair
\[
([a^{i},b^{i}]\mid 1\leq i\leq \ell),
([x^{j}_{1},y^{j}_{1}]\times \cdots \times [x^{j}_{n-1},y^{j}_{n-1}]
\mid 1\leq j\leq m)
\]
to the sequence of sub-cubes of $[0,1]^{n}$,
\[
[a^{i},b^{i}]\times [x^{j}_{1},y^{j},1]\times \cdots \times
[x^{j}_{n-1},y^{j}_{n-1}],
\]
(for $1\leq i\leq \ell$, $1\leq j\leq m$), labelled in lexicographical
order in $(i,j)$.  As an abbreviation of this notation, write
$\rho_{1}\colon \LC{n-1}(m)\to \LC{n}(m)$ for $\rho(\ie,-)$, where $\ie$
denotes the identity element of $\AC(1)$,
\begin{multline*}
\rho_{1}\left(
([x^{j}_{1},y^{j}_{1}]\times \cdots \times [x^{j}_{n-1},y^{j}_{n-1}]
\mid 1\leq j\leq m)\right)\\
=([0,1]\times [x^{j}_{1},y^{j}_{1}]\times \cdots 
\times [x^{j}_{n-1},y^{j}_{n-1}]
\mid 1\leq j\leq m) \in \LC{n}(m).
\end{multline*}
Since for any element $\cubec$ of $\LC{n-1}(m)$, $\rho_{1}(\cubec)$ is an
element of $\LC{n}(m)$, it specifies a map of $S$-modules
$\cim{\rho_{1}(\cubec)}\colon R^{(m)}\to 
R$.  The key fact we need is the following.

\begin{prop}
For any $\cubec$ in $\LC{n-1}(m)$, the map $\cim{\rho_{1}(\cubec)}\colon R^{(m)}\to R$ induced by
$\rho_{1}(\cubec)$ is a map of $\AC$-algebras.
\end{prop}

The proof consists of observing that for any $\cubea$ in $\AC(\ell)$,
both composites in the diagram
\[
\xymatrix@C+2pc{%
(R^{(m)})^{(\ell)}\ar[r]^{(\cim{\rho_{1}(\cubec)})^{(\ell)}}
\ar[d]_{\cubea}
&R^{(\ell)}\ar[d]^{\cubea}\\
R^{(m)}\ar[r]_{\cim{\rho_{1}(\cubec)}}&R
}
\]
can be identified as the map $\cim{\rho(\cubea,\cubec)}\colon R^{(\ell
m)}\to R$ under the isomorphism $R^{(\ell m)}\iso (R^{(m)})^{(\ell)}$
using the implicit lexicographical order.

Associated to the map of $\AC$-algebras $\cim{\rho_{1}(\cubec)}\colon
R^{(m)}\to R$, we get a map of enveloping algebras $U(R^{(m)})\to
UR$.  Concretely, in terms of the models $A$ of the previous section,
this is induced by the map
\[
D_{+}\sma R^{(m)}\to D_{+}\sma R
\]
that performs $\cim{\rho_{1}(\cubec)}$ on the $R$ factors and the identity
on $D$.   More generally, for any space $X$ and map $f\colon X\to
\LC{n-1}(m)$, we get a 
family of maps of $S$-algebras $U(R^{(m)})\to UR$, which by neglect of
structure gives a family of $(UR,U(R^{m}))$-bimodule structures on
$UR$, or equivalently, a $(UR,U(R^{m}))$-bimodule structure on $UR \sma
X_{+}$.  Concretely, the right $U(R^{(m)})$-action map is induced by
the map
\[
(D_{+}\sma R)\sma X_{+} \sma_{S} (D_{+}\sma R^{(m)})\to (D_{+}\sma
R)\sma_{S} (D_{+}\sma_{S}R)\to D_{+}\sma R,
\]
which is the composite of the map $X_{+}\sma R^{(m)}\to R$ induced by
$\rho_{1}(f)$ and the multiplication on $D_{+}\sma R$.

\begin{notn}\label{notnURf}
For $f\colon X\to \LC{n-1}(m)$, write $URf$ for $UR\sma X_{+}$ with
the\break $(UR, U(R^{(m)}))$-bimodule structure above.
\end{notn}

We have a canonical map of $S$-algebras $U(R^{(m)})\to (UR)^{(m)}$,
which formally is induced by the identification of $(UR)^{(m)}$ as the
left module enveloping algebra of $R^{(m)}$ as an $\AC^{m}$-algebra.
More concretely, it is induced by the map
\[
D_{+}\sma R^{(m)}\to (D_{+}\sma R)^{(m)},
\]
which performs the diagonal map on $D$.  We use this to regard the
smash product over $S$ of left $UR$-modules
\[
M_{1}\sma_{S}\cdots \sma_{S}M_{m}
\]
as a left $U(R^{(m)})$-module.  In this case, the left
$U(R^{(m)})$-module structure on $M_{1}\sma_{S}\cdots 
\sma_{S}M_{m}$ is induced by the diagonal map on $\AC$ and the left
$R$-module structure maps on the $M_{i}$:
\begin{multline*}
\AC(j+1)_{+}\sma (R^{(m)})^{(j)} \sma_{S} M_{1}\sma_{S}\cdots
\sma_{S}M_{m}\\
\shoveleft
\qquad \to \AC(j+1)^{m}_{+} \sma (R^{(m)})^{(j)} \sma_{S} M_{1}\sma_{S}\cdots
\sma_{S}M_{m}\\
\iso (\AC(j+1)_{+}\sma R^{(j)} \sma_{S}M_{1})\sma_{S}\cdots\sma_{S}
(\AC(j+1)_{+}\sma R^{(j)} \sma_{S}M_{m})\\
\to M_{1}\sma_{S}\cdots \sma_{S}M_{m}.
\end{multline*}

\begin{cons}\label{consopn}
For $f\colon X\to \LC{n-1}$, let $\opn{f}$ be the functor
$(\aM_{R})^{m}\to \aM_{R}$ defined by
\[
\opn{f}(M_{1},\ldots,M_{m})=URf\sma_{U(R^{(m)})}
(M_{1}\sma_{S}\cdots \sma_{S}M_{m}).
\]
\end{cons}

Having constructed the point-set operations, we now study them from
the perspective of the homotopy theory of $R$-modules.
Following Convention~\ref{convUR}, we understand homotopical concepts
in $R$-modules in terms of $UR$-modules.
Here we begin to take advantage of the technical properties of EKMM
$S$-modules: Because weak equivalences between cofibrant $R$-modules
are homotopy equivalences, and because topologically enriched functors
preserve homotopies, left derived functors of topologically enriched functors
always exist and are formed by applying the point-set functor to a
cofibrant approximation.  Equivalently, and more conveniently for us,
we can work in terms of $R$-modules that are homotopy equivalent to
cofibrant $R$-modules.  We use the following terminology.

\begin{defn}
A \term{homotopy cofibrant} $R$-module is an $R$-module that is
homotopy equivalent to a cofibrant $R$-module, or equivalently
\cite[VII.4.15]{ekmm}, homotopy equivalent to a cell $R$-module
\cite[III\S2]{ekmm}. 
\end{defn}

The $E_{n}$ interchange operations $\opn{f}$ are topologically
enriched, and in fact are enriched over $S$-modules as functors of
several variables.  Thus, their left derived functors exist and are
calculated by homotopy cofibrant approximation.  In fact, the
$S$-module enriched left derived functors \cite[\S5]{lm2} exist.

\begin{prop}\label{propderived}
For any $f\colon X\to \LC{n-1}(m)$, the left derived functor of
$\opn{f}$ exists, is computed by approximating by a weakly equivalent
homotopy cofibrant object, and
is enriched over the stable category.
\end{prop}

The main tool we have to study the homotopy theory of the $E_{n}$
interchange operations is the following lemma proved at the end of the
section.

\begin{lem}\label{lemcompu}
The canonical map $U(R^{(m)})\to (UR)^{(m)}$ is a homotopy equivalence
of left $U(R^{(m)})$-modules.
\end{lem}

For a cell $(UR)^{(m)}$-module $M$, applying the previous lemma
inductively, we see that $M$ is homotopy cofibrant as a
$U(R^{(m)})$-module.  This implies the following proposition.

\begin{prop}\label{propcompcof}
Let $M$ be a $(UR)^{(m)}$-module.  If $M$ is homotopy cofibrant as a
$(UR)^{(m)}$-module, then it is homotopy cofibrant as a
$U(R^{(m)})$-module. 
\end{prop}

We can now prove Theorems~\ref{thmopcomp} and~\ref{thmhocof}.

\begin{proof}[Proof of Theorem~\ref{thmopcomp}]
Write $M$ for the left $(UR)^{(j)}$-module $M_{1}\sma_{S}\cdots
\sma_{S}M_{j}$ in the statement.  The map in question is induced by
applying $(-)\sma_{(UR)^{(j)}}M$ to 
the map of $(UR,(UR)^{(j)})$-bimodules
\begin{multline*}
URf\sma_{U(R^{(m)})}(URg_{1}\sma_{S}\cdots \sma_{S}URg_{m})
\sma_{U(R^{(j_{1})})\sma_{S}\cdots \sma_{S}U(R^{(j_{m})})}
(UR)^{(j)}\\
\to
UR(f\circ (g_{1},\ldots,g_{m}))\sma_{U(R^{(j)})}(UR)^{(j)}.
\end{multline*}
Since by hypothesis, $M$ is a cofibrant left $(UR)^{(j)}$-module, it
suffices to show that the map above is a weak equivalence.  By
Lemma~\ref{lemcompu}, it suffices to show that the map
\[
URf\sma_{U(R^{(m)})}(URg_{1}\sma_{S}\cdots \sma_{S}URg)
\to
UR(f\circ (g_{1},\ldots,g_{m}))
\]
is a weak equivalence.  Since this is the map
\[
UR \sma X_{+} \sma_{U(R^{(m)})} 
(UR)^{(m)} \sma (Y_{1}\times \cdots \times Y_{m})_{+}
\to
UR \sma (X\times Y_{1}\times \cdots \times Y_{m})_{+},
\]
we see that it is a weak equivalence by applying
Lemma~\ref{lemcompu} a second time.
\end{proof}

\begin{proof}[Proof of Theorem~\ref{thmhocof}]
Applying Proposition~\ref{propcompcof}, we can choose a cell
$U(R^{(m)})$-module $M$ homotopy equivalent to $M_{1}\sma_{S}\cdots
\sma_{S}M_{m}$, and then it suffices to show that
$URf\sma_{U(R^{(m)})}M$ is homotopy cofibrant.  Working inductively
with the cell structure,
it suffices to check the case when $M$ is a single cell
$U(R^{(m)})\sma_{S}S^{n}_{S}$ (in the notation of \cite[II.1.7]{ekmm}).
In this case,
\[
URf\sma_{U(R^{(m)})}(U(R^{(m)}) \sma_{S}S^{n}_{S})
\iso URf \sma S^{n}_{S}=(UR\sma X_{+})\sma_{S}S^{n}_{S}
\iso UR \sma_{S} S^{n}_{S}\sma X_{+}
\]
is homotopy cofibrant.
\end{proof}

We close this section with the proof of Lemma~\ref{lemcompu}.  The
proof requires the construction of $U(R^{(m)})$ as $A(R^{(m)})$ in 
Section~\ref{seclmea}.  We begin by describing maps and homotopies on
$D$.  

Write $\Delta$ for the diagonal map $D\to D^{m}$ and consider the map
$g\colon D^{m}\to D$ defined by 
\[
g\colon (\cubec_{1},\ldots,\cubec_{m})=
(([0,c_{1}],[c_{1},d_{1}]),\ldots,([0,c_{m}],[c_{m},d_{m}]))
\mapsto ([0,c],[c,d])
\]
where $c=\max\{c_{1},\ldots,c_{m}\}$ and
$d-c=\min\{d_{1}-c_{1},\ldots,d_{m}-c_{m}\}$.  We have that $g\circ
\Delta$ is the identity on $D$.  We obtain a homotopy $h_{t}$ on
$D^{m}$ from $\Delta \circ g$ to the identity defined
by the linear homotopy in each coordinate
\[
h_{t}(\cubec_{1},\ldots,\cubec_{m})_{i}=([0,c+t(c_{i}-c)],[c+t(c_{i}-c),d+t(d_{i}-d)]).
\]
Note that $d+t(d_{i}-d)>c+t(c_{i}-c)$ since $(1-t)(d-c)+t(d_{i}-c_{i})>0$.
Analogous formulas define maps and
homotopies when one or more factors of $D$ are replaced by $\bar D$
(with the analogue of $g$ landing in $\bar D$ when all factors are
$\bar D$).

Next we see how the homotopies interact with the maps $p$ and $q$ in
the construction of $A$.  For $\cubea=([0,a],[a,b])$ in $D$, we have
\[
p(\cubea,h_{t}(\cubec_{1},\ldots,\cubec_{m})_{i})=([0,x],[x,y]),
\]
where 
\begin{multline*}
x=a+(b-a)(c+t(c_{i}-c))=a+(b-a)c+t(b-a)(c_{i}-c)\\
=a+(b-a)c + t\bigl( (a+(b-a)c_{i})-(a+(b-a)c) \bigr)
\end{multline*}
and 
\begin{multline*}
y=a+(b-a)(d+t(d_{i}-d))=a+(b-a)d+t(b-a)(d_{i}-d)\\
=a+(b-a)d + t\bigl( (a+(b-a)d_{i})-(a+(b-a)d) \bigr).
\end{multline*}
Since $\max\{a+(b-a)c_{i}\}$ is $a+(b-a)c$ and
$\min\{(a+(b-a)d_{i})-(a+(b-a)c_{i})\}$ is $(b-a)(d-c)$, we see that 
\[
p(\cubea,h_{t}(\cubec_{1},\ldots,\cubec_{m})_{i})
=h_{t}(p(\cubea,\cubec_{1}),\ldots,p(\cubea,\cubec_{m}))_{i}.
\]
Likewise, since
\[
\frac a{(a+(b-a)(c+t(c_{i}-c)))}=
\frac a{\left(a+(b-a)c+t\bigl( (a+(b-a)c_{i})-(a+(b-a)c) \bigr)\right)},
\]
we have
\[
q(\cubea,h_{t}(\cubec_{1},\ldots,\cubec_{m})_{i})=
h_{t}(q(\cubea,\cubec_{1}),\ldots,q(\cubea,\cubec_{m}))_{i}.
\]

Putting this together with $R^{(m)}$, we get a map
\[
g\colon D^{m}_{+}\sma R^{(m)}\to D_{+}\sma R^{(m)}
\]
and a homotopy
\[
h_{t}\colon D^{m}_{+}\sma R^{(m)}\to D^{m}_{+}\sma R^{(m)}.
\]
The formulas above imply that these are compatible with the left
action of $D_{+}\sma R^{(m)}$.  Moreover, these are compatible with
the analogous maps obtained by replacing one or more factors of $D$ by
$\bar D$ and the corresponding factor of $R$ with $S$.  Passing to
iterated pushouts, we get a map
\[
g\colon (UR)^{(m)}\to U(R^{(m)})
\]
and a homotopy 
\[
h_{t}\colon (UR)^{(m)}\to (UR)^{(m)}
\]
compatible with the left $U(R^{(m)})$-action.

\section{Proof of the Main Theorem}
\label{secsma}

In this section, we prove the Main Theorem, which amounts to
specifying constructions and verifying coherence diagrams.  For the case
of an $E_{2}$ algebra, we construct the smash product
in~\ref{secconssma}, the right adjoint function modules
in~\ref{secconsfun}, and compare the categories of left and right
modules in~\ref{secright}.  We construct the unit and associativity
isomorphisms in~\ref{secconsunit} and~\ref{secconsass},
and prove the unit and associativity 
coherence in~\ref{seccoher}.  For the $E_{3}$ case, we construct the
braid isomorphism and prove its coherence in~\ref{secbraid}, and for
the $E_{4}$ case, we show that the braid isomorphism is a symmetry
isomorphism in~\ref{secsym}. 

\subsection{The Smash Product}
\label{secconssma}

Let $\mu$ be the element $([0,1/2],[1/2,1])$ in $\LC{1}(2)$.  For left
$R$-modules $M$,$N$, define 
\[
M\sma_{R}N = \opn{\mu}(M,N).
\]
Let $\sma_{R}\colon \aD_{R}\times \aD_{R}\to \aD_{R}$ be the left
derived functor.

\subsection{The Function Modules}
\label{secconsfun}

For a left $R$-module $M$, let
\[
M^{\ell}=UR\mu\sma_{U(R^{(2)})}(UR\sma_{S}M)
\qquad \text{and}\qquad
M^{r}=UR\mu\sma_{U(R^{(2)})}(M\sma_{S}UR)
\]
in the notation of~\ref{notnURf}.
These are $(UR,UR)$-bimodules using the left $UR$-module structure on
$UR\mu$ and the right $UR$-module structure on $UR$.  For example,
\begin{align*}
UR \sma_{S} M^{\ell} \sma_{S} UR
&=UR\sma_{S}(UR\mu\sma_{U(R^{(2)})}(UR\sma_{S}M))\sma_{S}UR\\
&\iso (UR \sma_{S} UR\mu)\sma_{U(R^{(2)})}((UR\sma_{S}UR)\sma_{S}M)\\
&\to UR\mu\sma_{U(R^{(2)})}(UR\sma_{S}M) = M^{\ell}.
\end{align*}
Clearly, the
functors 
\[
F_{UR}(M^{\ell},-)\qquad \text{and}\qquad F_{UR}(M^{r},-)
\]
are right adjoint to the point-set functors $(-)\sma_{R}M$ and
$M\sma_{R}(-)$ defined above.  Now assume $M$ is homotopy
cofibrant. Then $M^{\ell}$ and $M^{r}$ are homotopy cofibrant as left
$UR$-modules, and so these functors preserve weak equivalences between
arbitrary $UR$-modules; therefore, their right derived functors
exist.  Since for any homotopy cofibrant $N$, $N\sma_{R}M$ and
$M\sma_{R}N$ are homotopy cofibrant, an easy check shows that the
right derived functors of $F_{UR}(M^{\ell},-)$ and $F_{UR}(M^{r},-)$
remain adjoint to the left derived functors of $(-)\sma_{R}M$ and
$M\sma_{R}(-)$.

\subsection{Comparison of Left and Right Modules}
\label{secright}

Forgetting the left $UR$-module structure on $M^{r}$ defines a functor
$r$ from $\aM_{UR}$ to $\aM_{UR^{\op}}$, and a derived functor from
$\aD_{R}$ to $\aD_{R^{\op}}$.  By construction, (the underlying
$S$-module of) the smash product
above is the composite of $r$ with the balanced product of a left and
right $UR$-module
\[
\opn{\mu}(M,N)=rM \sma_{UR} N.
\]
To see that $r$ induces an equivalence on derived categories, we can
rewrite $r$ as
\[
rM=(UR\mu \sma_{U(R^{(2)})}(UR\sma_{S}UR)) \sma_{UR} M.
\]
Writing $W$ for $UR\mu \sma_{U(R^{(2)})}(UR\sma_{S}UR)$,
we can identify the derived functor as 
\[
r(-)=\Tor_{UR}({}_{UR^{\op}}W,-)
\]
in the notation of \cite{lm2}.  Applying \cite[8.5]{lm2}, we see that
the right adjoint\break $\Ext_{UR^{\op}}(W_{UR},-)$ exists. Since $W$ is weakly
equivalent to $UR$ in each of its right $UR$-module structures, both
derived functors are naturally isomorphic to the identity on the
underlying $S$-modules.  In particular, it follows that the unit and
counit of the derived adjunction are isomorphisms and these functors
are inverse equivalences.

\subsection{The Unit Isomorphisms}
\label{secconsunit}

Since $UR\to R$ is a weak equivalence, we have $UR\sma_{S} S_{S}\to R$
as a cofibrant approximation.  Let 
\[
\mu_{0}^{\ell}=([1/2,1])\qquad \text{and}\qquad \mu_{0}^{r}=([0,1/2]),
\]
elements of $\LC{1}(1)$.  The maps of $\AC$-algebras 
\[
i_{1}\colon R=S\sma_{S}R\to R\sma_{S}R\qquad \text{and}\qquad
i_{2}\colon R=R\sma_{S}S\to R\sma_{S}R
\]
allow us to regard $UR\mu$ as a $(UR,UR)$-bimodule two different ways,
and we have canonical isomorphisms of left $UR$-modules
\[
UR\mu\sma_{UR,i_{1}}M\iso UR\mu_{0}^{\ell}\sma_{UR}M
\qquad \text{and}\qquad 
UR\mu\sma_{UR,i_{2}}M\iso UR\mu_{0}^{r}\sma_{UR}M.
\]
Letting $\eta^{\ell}$ and $\eta^{r}$
denote the linear paths in $\LC{1}$ from $\mu_{0}^{\ell}$ and $\mu_{0}^{r}$ to
$\ie=([0,1])$, we then have natural maps
\begin{gather*}
\opn{\mu}(UR\sma S_{S},M)
\from \opn{\mu_{0}^{\ell}}(M)\sma_{S}S_{S}\to \opn{\eta^{\ell}}(M)
\from \opn{\{\ie\}}(M)=M\\
\opn{\mu}(M,UR\sma S_{S})
\from \opn{\mu_{0}^{r}}(M)\sma_{S}S_{S}\to \opn{\eta^{r}}(M)
\from \opn{\{\ie\}}(M)=M,
\end{gather*}
in which all maps are weak equivalences when $M$ is homotopy
cofibrant.  These are the left 
and right unit isomorphisms $\lambda$ and $\rho$. 

\subsection{The Associativity Isomorphism}
\label{secconsass}
As indicated in Section~\ref{secoutline}, for an appropriate path
$\alpha$ in $\LC{1}(3)$, the associativity
isomorphism is the zigzag~\eqref{eqconsass},
\[
\xymatrix@-1pc{%
\opn{\mu\circ_{1}\mu}(L,M,N)\ar[d]\ar[r]
&\opn{\alpha}(L,M,N)
&\opn{\mu\circ_{2}\mu}(L,M,N)\ar[l]\ar[d]\\
\opn{\mu}(\opn{\mu}(L,M),N)
&&\opn{\mu}(L,\opn{\mu}(M,N)),
}
\]
in which all the maps are weak equivalences when $L$, $M$, and $N$ are
homotopy cofibrant.  

\subsection{The Coherence Diagrams}
\label{seccoher}

For coherence of associativity, we want to show that the pentagon
diagram in $\aD_{R}$ 
\[
\xymatrix@C-4pc{%
&(K\sma_{R}L)\sma_{R}(M\sma_{R}N)\ar[dr]\\
((K\sma_{R}L)\sma_{R}M)\sma_{R}N\ar[ur]\ar[d]&&
K\sma_{R}(L\sma_{R}(M\sma_{R}N))\\
(K\sma_{R}(L\sma_{R}M))\sma_{R}N\ar[rr]
&&K\sma_{R}((L\sma_{R}M)\sma_{R}N)\ar[u]
}
\]
commutes.  The paths
\[
\alpha \circ_{1}\mu,\qquad
\mu\circ_{1}\alpha,\qquad 
\alpha \circ_{2}\mu,\qquad
\mu\circ_{2}\alpha,\qquad
\alpha \circ_{3}\mu
\]
specify a map from the boundary of the pentagon into $\LC{1}(4)$,
which can be filled in to a map $\pi$ from the pentagon into
$\LC{1}(4)$ by the contractibility of the components of $\LC{1}(4)$ (or by
making explicit choices).  We then have the following commutative
diagram of weak equivalences of $(UR,U(R^{4}))$-bimodules.
\[
\xymatrix@C-1pc{%
&UR(\mu \circ (\mu ,\mu))\ar[dl]\ar[dr]\ar[dd]\\
UR(\alpha\circ_{1}\mu)\ar[dr]&&UR(\alpha\circ_{3}\mu)\ar[dl]\\
UR(\mu\circ_{1}(\mu\circ_{1}\mu))\ar[u]\ar[d]\ar[r]
&UR\pi&UR(\mu\circ_{2}(\mu\circ_{2}\mu))\ar[u]\ar[d]\ar[l]\\
UR(\mu \circ_{1}\alpha)\ar[ur]&&UR(\mu\circ_{2}\alpha)\ar[ul]\\
UR(\mu \circ_{1}(\mu \circ_{2}\mu))\ar[u]\ar[r]\ar[uur]
&UR(\alpha \circ_{2}\mu)\ar[uu]
&UR(\mu \circ_{2}(\mu \circ_{1}\mu))\ar[l]\ar[u]\ar[uul]
}
\]
We see that both composites of derived functors
\[
\opn{\mu}\circ_{1}(\opn{\mu}\circ_{1}\opn{\mu}) \to 
\opn{\mu}\circ_{2}(\opn{\mu}\circ_{2}\opn{\mu})
\]
in $\aD_{R}$ are represented by the zigzag
\[
\opn{\mu}\circ_{1} (\opn{\mu}\circ_{1}\opn{\mu})
\from \opn{\mu\circ_{1}(\mu\circ_{1}\mu)}
\to \opn{\pi}\from \opn{\mu \circ_{2}(\mu \circ_{2}\mu)}
\to \opn{\mu}\circ_{2}(\opn{\mu}\circ_{2}\opn{\mu}),
\]
and so coincide.  It follows that the associativity coherence pentagon
in $\aD_{R}$ commutes.

For the coherence of the unit, we want to show that the triangle
diagram in $\aD_{R}$ 
\[
\xymatrix@C-2pc{%
M\sma_{R}(R\sma_{R}N)\ar[rr]^{\alpha}\ar[dr]_{\id\sma \lambda}
&&(M\sma_{R}R)\sma_{R}N\ar[dl]^{\rho \sma \id}\\
&M\sma_{R}N
}
\]
commutes.  Letting $\ze$ denote the unique element of $\LC{1}(0)$, the paths
\[
\mu \circ_{2} \eta^{\ell},\qquad
\mu \circ_{1} \eta^{r},\qquad
\alpha \circ (\ie,\ze,\ie)
\]
specify a map from the boundary of the triangle to $\LC{1}(2)$ that
fills in by contractibility.  An argument like the previous one then
shows that the unit triangle in $\aD_{R}$ commutes.

\subsection{The Braid Isomorphism}
\label{secbraid}

We now assume that $R$ is a $\LC{3}$-algebra.  Since $\LC{2}(2)$ is
connected, we can choose a path $\sigma$ from 
\[
\mu=([0,1/2]\times [0,1], [1/2,1]\times [0,1])
\]
to
\[
\mu\tau = ([1/2,1]\times [0,1], [0,1/2]\times [0,1]).
\]
(Such a path is illustrated in Section~\ref{secoutline}.)  We then get
a braid isomorphism in $\aD_{R}$ from the zigzag
\[
\opn{\mu}(M,N)\to \opn{\sigma}(M,N) \from \opn{\mu\tau}(M,N)
\]
and the isomorphism $\opn{\mu\tau}(M,N)\iso \opn{\mu}(N,M)$ induced
(upon passing to coequalizers) by the isomorphism 
\[
UR\sma_{S} M \sma_{S}N \to UR\sma_{S} N \sma_{S} M.
\]
We need to show that hexagon diagram in $\aD_{R}$
\[
\xymatrix@C-2pc{%
&(M\sma_{R}L)\sma_{R}N\ar[rr]^{\alpha}
&\quad&M\sma_{R}(L\sma_{R}N)\ar[dr]^{\id\sma \sigma }\\
(L\sma_{R}M)\sma_{R}N\ar[ur]^{\sigma\sma \id}\ar[dr]_{\alpha}
&&&&M\sma_{R}(N\sma_{R}L)\\
&L\sma_{R}(M\sma_{R}N)\ar[rr]_{\sigma}
&&(M\sma_{R}N)\sma_{R}L\ar[ur]_{\alpha}
}
\]
and the analogous hexagon with $\sigma$ replaced with its inverse (or
equivalently, $\alpha$ replaced by its inverse) commute. The paths
\[
\mu\circ_{1}\sigma,\qquad
\alpha,\qquad
\mu\circ_{2}\sigma,\qquad
\alpha, \qquad
\sigma\circ_{1}\mu,\qquad
\alpha
\]
join together to define a map from the boundary of a hexagon into
$\LC{2}(3)$.  The fundamental group of $\LC{2}(3)$ is the braid group
$B_{3}$ on 3 strands, and the braid relation on $\pi_{1}(\LC{2}(3))$
implies that 
this can be filled in to a map from the hexagon.  The remainder of the
proof follows just as in the arguments in the previous subsection, and
the other case is similar.

\subsection{The Symmetry Isomorphism}
\label{secsym}

We now assume that $R$ is a $\LC{4}$-algebra.  We want to show that
$\sigma_{M,N}=\sigma^{-1}_{N,M}$ in $\aD_{R}$, that is, that the
composite map in $\aD_{R}$
\[
\xymatrix@C-1pc{%
M\sma_{R}N\ar[r]^{\sigma}&N\sma_{R}M\ar[r]^{\sigma}&M\sma_{R}N
}
\]
is the identity.  This follows from the fact that the loop obtained
from the paths
\[
\sigma, \qquad \sigma \tau
\]
in $\LC{2}(2)$ is contractible in the simply connected space
$\LC{3}(2)$.  (This loop is not 
contractible in $\LC{2}(2)$, but generates
$\pi_{1}(\LC{2}(2))=B_{2}\iso \mathbb{Z}$.)

\section{The Lax Monoidal Smash Product}
\label{secrefine}

Previous sections have concentrated on the monoidal structure on the
derived category of left $R$-modules for a $\LC{2}$-algebra $R$.  In
this section, we study the structure that arises on the point-set
category of $R$-modules.  Although we do not get a true monoidal
structure, we do get some kind of weaker structure.  The purpose of this
section is to describe this structure and to outline
an approach to constructing it.

We organize the discussion in terms of the \term{lax monoidal
structures} of \cite[3.1.1]{leinsterbook}.  Such a
structure on a category $\aM$ consists of functors
\[
\botimes_{n}\colon \aM^{n}\to \aM
\]
for $n$ a natural number (including zero) and natural transformations
$\gamma,\iota$ 
satisfying certain coherence conditions.  These coherence conditions
are most concisely specified in terms of trees and edge contractions:
An arbitrary composite of the functors $\botimes$ can be viewed as a
planar tree with leaves labelled either by an object of $\aM$ or by
$\botimes_{0}$.  The natural transformations $\gamma$ compose
\begin{multline*}
\botimes_{m}(M_{1},\ldots,M_{i-1},
\botimes_{n}(N_{1},\ldots,N_{n}),M_{i+1},\ldots,M_{m})\\
\to \botimes_{m+n-1}(M_{1},\ldots,M_{i+1},
N_{1},\ldots,N_{n},M_{i+1},\ldots,M_{m}),
\end{multline*}
and so contract an
internal edge or edge ending in a leaf labelled $\botimes_{0}$; the natural
transformation $\iota$ is a map 
\[
M\to \botimes_{1}M
\]
and so is essentially an edge insertion, converting a node into two
nodes with an edge connecting them.  The coherence conditions are that
(1) all sequences of edge contractions that take a given
planar labelled tree to another given planar labelled tree performs
the same natural transformation, and (2) an edge insertion followed by
an edge contraction of the inserted edge is the identity.  See
\cite[3.1.1]{leinsterbook} for a formulation not involving tree
operations. 

A fundamental property of the theory of lax monoidal categories is
Theorem~3.1.6 of 
\cite{leinsterbook}: A lax monoidal category in which the natural
transformations $\gamma$ and $\iota$ are isomorphisms is equivalent to
a (strict) monoidal category.  From the perspective of homotopy
theory, we can view a lax monoidal category where the natural
transformations $\gamma$ and $\iota$ are weak equivalences (say, after
restricting to homotopy cofibrant objects) as an 
up-to-coherent-homotopy version of a monoidal category.

In the context of an $E_{2}$ algebra $R$, to give the idea of how to
construct the lax monoidal 
structure on $\aM_{R}$, we begin with a non-unital version, omitting
$\botimes_{0}$.  For this version of the construction, we can also take
$\botimes_{1}$ and $\iota$ to be the identity functor and map.  Then
we only need to treat $\botimes_{n}$ for $n\geq 2$.  The composite
operations are then the planar trees labelled by objects of $\aM_{R}$
where all internal nodes have valence $2$ or more.  Choosing a map of
operads $\phi$ from the Stasheff operad $\SK$ to $\AC$, we construct
$\botimes_{n}$ inductively as follows. Writing $\phi_{n}$ for the map
$\SK(n)\to \AC(n)$, $\phi_{2}$ is the inclusion of a point in
$\AC(2)$, and we let
$\botimes_{2}(M_{1},M_{2})=\opn{\phi_{2}}(M_{1},M_{2})$.  We have that
$\SK(2)$ is an interval and $\phi_{3}$ is a path between
$\phi_{2}\circ_{1}\phi_{2}$ and $\phi_{2}\circ_{2}\phi_{2}$; we define
$\botimes_{3}$ as the colimit of the diagram
\[
\xymatrix@-1pc{%
\opn{\phi_{2}\circ_{1}\phi_{2}}(M_{1},M_{2},M_{3})\ar[d]\ar[r]
&\opn{\phi_{3}}(M_{1},M_{2},M_{3})
&\opn{\phi_{2}\circ_{2}\phi_{2}}(M_{1},M_{2},M_{3})\ar[d]\ar[l]\\
\botimes_{2}(\botimes_{2}(M_{1},M_{2}),M_{3})
&&\botimes_{2}(M_{1},\botimes_{2}(M_{2},M_{3})),
}
\]
which we can identify as the pushout of the maps 
\begin{gather*}
\opn{\phi_{2}\circ_{1}\phi_{2}}(M_{1},M_{2},M_{3})\to
\opn{\phi_{2}}(\opn{\phi_{2}}(M_{1},M_{2}),M_{3}),\\
\opn{\phi_{2}\circ_{2}\phi_{2}}(M_{1},M_{2},M_{3})
\to \opn{\phi_{2}}(M_{1},\opn{\phi_{2}}(M_{2},M_{3}))
\end{gather*}
over the Hurewicz cofibration
\[
\opn{\phi_{2}\circ_{1}\phi_{2}}(M_{1},M_{2},M_{3})
\vee
\opn{\phi_{2}\circ_{2}\phi_{2}}(M_{1},M_{2},M_{3})
\to
\opn{\phi_{3}}(M_{1},M_{2},M_{3}).
\]
In general, the polytope $\SK(n)$ has a sub-face for each planar tree
with all internal nodes of valence $2$ or more.  Thus, the
boundary consists of all formal compositions (of total valence $n$) of
all lower valence polytopes.  We can glue together the corresponding
compositions of the operations $\botimes_{m}$ to form an operation
``$\botimes_{\partial \SK(n)}$''.  Writing $\partial \phi_{n}$ for the
restriction of $\phi_{n}$ to the boundary of $\SK(n)$, we have a
natural transformation
\[
\opn{\partial\phi_{n}}(M_{1},\ldots,M_{n})\to 
\botimes_{\partial \SK(n)}(M_{1},\ldots,M_{n}),
\]
since we can identify $\opn{\partial \phi}$ as the colimit obtained by
gluing 
the corresponding operations $\opn{f}$ obtained by composing in the
operad. Moreover, 
both colimits are formed by corresponding
iterated pushouts along Hurewicz cofibrations; by induction, when the
modules $M_{1},\ldots, M_{n}$ are homotopy cofibrant, the comparison
map on each formal composition in the boundary is a weak equivalence,
and so the natural transformation above on their colimits is a weak
equivalence. We define
$\botimes_{n}(M_{1},\ldots,M_{n})$ by the pushout diagram
\[
\xymatrix{%
\opn{\partial\phi_{n}}(M_{1},\ldots,M_{n})\ar[r]\ar[d]&
\opn{\phi_{n}}(M_{1},\ldots,M_{n})\ar[d]\\
\botimes_{\partial \SK(n)}(M_{1},\ldots,M_{n})\ar[r]
&\botimes_{n}(M_{1},\ldots,M_{n}).
}
\]
The compositions $\gamma$ are then Hurewicz cofibrations, and are weak
equivalences when the modules $M_{1},\ldots, M_{n}$ are homotopy
cofibrant.

The previous construction used the interpretation of the cells of the
Stasheff operad in terms of trees, or equivalently, the fact that the
Stasheff operad is the cofibrant operad on one cell in each valence
(or ``arity'') $n\geq 2$.  To put the unit in, we need to use a cofibrant
$A_{\infty}$ operad operad $\SK^{u}$ having $\SK^{u}(0)$ contractible
instead of empty.  Using generating cells in valence zero reflecting the
structure of the unit maps in Section~\ref{secconsunit}, the
construction above then generalizes to produce a lax monoidal structure
with $\SK^{u}(n)$ (rather than $\SK(n)$) parametrizing the
construction of $\botimes_{n}$.  We omit the remaining details.

At the cost of weakening the point-set structure further, we get a
structure even closer to the structure of $E_{n}$ interchange
operations from Section~\ref{secpair}.  We introduce the following
``partial'' version of a lax monoidal category.  Again, this is
easiest to explain in terms of planar trees.  We write $\aT_{n}$ for
the partially ordered set of planar trees with $n$ distinguished
leaves (terminal nodes), where we have a map $T\to T'$ in $\aT_{n}$
when $T'$ can be 
obtained from $T$ by contracting internal edges (edges that end in
an internal node) and/or edges ending in undistinguished leaves.  We
understand $\aT_{0}$ as the 
category with a single object (the empty tree) and morphism (the
identity). We have functors
\[
\circ_{i}\colon \aT_{m}\times \aT_{n}\to \aT_{m+n-1}
\]
that send $(T,T')$ to the tree that grafts $T'$ onto $T$ replacing
the $i$-th distinguished leaf (counting from left to right) if $T'$ is
not the empty tree, or makes the $i$-th distinguished leaf undistinguished if $T'$ is
the empty tree.

\begin{defn}
A \term{partial lax monoidal structure} on a category $\aM$ consists
of functors
\[
\botimes_{(-)}\colon \aT_{n}\times \aM^{n}\to \aM,
\]
natural transformations
\[
\eta \colon \botimes_{T\circ_{i}T'}\to \botimes_{T}\circ_{i}\botimes_{T'}
\]
and a natural transformation
\[
\iota \colon \Id \to \botimes_{S_{1}}
\]
(where $S_{1}$ is the star with one leaf), such that the transitivity
diagrams 
\[
\xymatrix{%
\botimes_{(T_{1}\circ_{i}T_{2})\circ_{j}T_{3}}\ar[r]^{=}\ar[d]_{\eta}
&\botimes_{T_{1}\circ_{i'}(T_{2}\circ_{j'}T_{3})}\ar[d]^{\eta}\\
\botimes_{T_{1}\circ_{i}T_{2}}\circ_{j}\botimes_{T_{3}}\ar[d]_{\eta}
&\botimes_{T_{1}}\circ_{i'}\botimes_{T_{2}\circ_{j'}T_{3}}\ar[d]^{\eta}\\
(\botimes_{T_{1}}\circ_{i}\botimes_{T_{2}})\circ_{j}\botimes_{T_{3}}\ar[r]_{=}
&\botimes_{T_{1}}\circ_{i'}(\botimes_{T_{2}}\circ_{j'}\botimes_{T_{3}})
}
\]
commute for all $i$, $j$ (and appropriate $i'$,$j'$), and the unit
diagrams
\[
\xymatrix{%
\botimes_{S_{1}\circ T}\ar[d]\ar[dr]^{\eta}
&&
\botimes_{T\circ_{i}S_{1}}\ar[d]\ar[dr]^{\eta}\\
\botimes_{T}\ar[r]_(.4){\iota}&\botimes_{S_{1}}\circ \botimes_{T}
&
\botimes_{T}\ar[r]_(.4){\botimes_{T}\iota}
&\botimes_{T}\circ_{i} \botimes_{S_{1}}
}
\]
commute for all $i$.
\end{defn}

A partial lax monoidal category for which the natural transformations
$\eta$ are isomorphisms is equivalent to a lax monoidal category: We
take $\botimes_{n}$ to be $\botimes_{S_{n}}$ for the stars
$S_{n}$. The natural transformations $\gamma$ are the composites
\[
\botimes_{m}\circ_{i}\botimes_{n}\xrightarrow{\eta^{-1}}
\botimes_{S_{m}\circ_{i}S_{n}}\to \botimes_{m+n-1},
\]
where the unlabelled arrow is the map induced by the edge contraction
$S_{m}\circ_{i}S_{n}\to S_{m+n-1}$.  A partial lax monoidal category
for which all the structure maps (the maps $\eta$, $\iota$, and the
maps induced by maps in $\aT_{n}$) are weak equivalences is then
another kind of up-to-coherent-homotopy version of a monoidal category.

In our context, we have the following result on the partial lax
monoidal structure on the category of left modules over an $E_{2}$ algebra $R$.
The unit of this structure will be the left module $UR$, which is not
cofibrant, but does have the property that $UR\sma_{S}S_{S}$ is
cofibrant.  In the following theorem, we say that an $R$-module is 
\term{nearly homotopy cofibrant} if $(-)\sma_{S}S_{S}$ makes it into a
homotopy cofibrant $R$-module.

\begin{thm}\label{thmplmcat}
For a $\LC{2}$-algebra $R$, $\aM_{R}$ is a partial lax monoidal
category.  This structure restricts to a partial lax monoidal
structure on the full subcategory of nearly homotopy cofibrant
$R$-modules; moreover, on this subcategory, the structure maps are
weak equivalences. 
\end{thm}

Given a tree $T$, write 
\[
\AC(T)=\AC(n_{1})\times \dotsb \times \AC(n_{r})
\]
where $n_{1},\dotsc,n_{r}$ are the valences of the internal nodes.
This, together with the operadic multiplication, makes $\AC$ into
functors on the categories $\aT_{n}$, 
as in~\cite[\S1]{gk}. We define 
\[
\botimes_{T}=\opn{\AC(T)},
\]
we take the maps $\eta$ to be the natural transformations as
constructed in~\eqref{eqopmap}, and we take $\iota$ to be the
inclusion $\Id=\opn{\{\ie\}}\to \opn{\AC(1)}=\botimes_{S_{1}}$.  The
commutativity of the diagrams is an easy check of the definitions.
As the constructions preserve homotopy equivalences and commute
with $(-)\sma_{S}S_{S}$, the structure restricts to a structure on
the full subcategory of nearly homotopy cofibrant modules by
Theorem~\ref{thmhocof}.  The weak equivalence assertion follows from
Theorem~\ref{thmopcomp} and its proof. 

The construction above is sufficient for Theorem~\ref{thmplmcat}, but
is only guaranteed to give the correct homotopy types on the nearly
homotopy cofibrant $R$-modules.  One could imagine constructing 
functors that are correct on a more general class of $R$-modules by
using a two-sided bar construction or homotopy universal left Kan
extension (HULK). We have not pursued this HULK smash product: in all
but the simplest cases, the HULK is very difficult to control.

\subsection*{Generalizations}

For an $E_{n}$ algebra $R$, $n>2$, one expects a point-set structure
on the category of $R$-modules reflecting the additional structure on $R$.
Using ideas of \cite{dunntensor} and \cite{iteratedmonoidal}, one 
possibility would be some kind of lax (or partial lax) iterated
monoidal category structure.  Alternatively, one can view a lax
monoidal structure on a category $\aC$ as a pseudo-functorial
map of operads of categories
\[
\Sigma \to \End(\aC),
\]
where $\Sigma$ denotes the associative algebra operad
$\Sigma(n)=\Sigma_{n}$ (viewed as a discrete category), and
$\End(\aC)$ denotes the endomorphism operad 
\[
\End(\aC)(n)=\Fun(\aC^{n},\aC)
\]
of functors $\aC^{n}\to \aC$ and natural transformations. Recent
development of the theory
of quasi-categories give an interpretation of $\End(\aC)$
as an operad in $(\infty,1)$-categories (using a simplicial
localization, singular complex, 
or homotopy coherent nerve construction); another formulation of the
expected structure would be a map 
\[
\LC{n-1}\to \End(\aM_{R})
\]
in an appropriate homotopy category of operads of
$(\infty,1)$-categories.  We intend these remarks as suggestive rather
than rigorous and offer no further details. 

\subsection*{Converse Conjectures}

In the context of stable homotopy categories, under suitable technical
hypotheses, the thick subcategory generated by a particular object $X$
is equivalent to the thick subcategory of small objects in the derived
category of the endomorphism ring spectrum $\End(X)$ of $X$.  The
converse conjecture in the introduction is based on the familiar
principle that structure 
on the derived category should reflect and be reflected by structure
on the ring.  

In the case of an $E_{2}$ ring spectrum $R$, we
have seen above that the derived category obtains a monoidal structure and
the point-set category obtains a weakened version of a monoidal
structure with the unit weakly equivalent to $R$.
Starting from the other side, given a category $\aC$ with an
appropriate notion of weak equivalence and an appropriate weakened
monoidal structure with unit $U$, consider the derived endomorphism
ring spectrum $\End(U)$.  Under suitable technical conditions, we can
construct $\Hom$ spectra of the appropriate homotopy type and
$\End(U)=\Hom(U,U)$ has an $S$-algebra structure under
composition (or partial or $A_{\infty}$ $S$-algebra structure
depending on how strictly the $\Hom$ spectra compose).  The weakened monoidal
structure gives us zigzags of weak equivalences
between $U'\otimes \dotsb \otimes U'$ and $U$ (where $U'$ may be 
a cofibrant approximation or similar homotopical replacement), and
this should (conjecturally) 
induce a second (partial and/or $A_{\infty}$) structure on $\End(U)$ that
satisfies an appropriate 
homotopy interchange property with respect to composition.  Together,
these can then be rectified to an $E_{2}$ ring spectrum structure.

In the special case of the monoidal category of bimodules over an
$S$-algebra, McClure and Smith \cite{McClureSmith} produced such an
$E_{2}$ structure, affirming the Deligne Hochschild cohomology
conjecture.  Since 2004, the author has been advertising the following
problem, generalizing the Deligne conjecture and providing a converse.

\begin{prob}
Formulate a point-set homotopy coherent $E_{n-1}$-monoidal structure
that arises on the category of modules over an $E_{n}$ ring
spectrum. Prove that for a category $\aC$ with such a structure, under
appropriate technical hypotheses, the derived endomorphism ring
spectrum $E$ of the unit is an $E_{n}$-algebra such that the induced
structure on the category of $E$-modules is compatible with the
original structure on $\aC$.
\end{prob}

This problem has since been solved by Clark Barwick \cite{HCA} and
David Gepner \cite{Gepner}; another statement of a version of the
result can be found in Lurie's DAG-VI~\cite[2.3.15]{DAGVI}. 

\section{The Moore Algebra}
\label{secmoore}

We close this paper with a brief note about the relationship between
the left module enveloping algebra of an $\AC$-algebra and the Moore
algebra.  While we can make sense of the left module enveloping
algebra for an algebra over an arbitrary non-$\Sigma$ operad, the
Moore algebra construction is specific to algebras over $\AC$: An
$\AC$-algebra $R$ has the same relationship to its Moore algebra $\MA$
as the based loop space of a topological space has to its Moore loop
space.

We begin with the construction of the Moore algebra.  For this, we let
$P=(0,\infty)\subset \bR$ and $\bar 
P=[0,\infty)\subset \bR$ denote the positive real numbers and
non-negative real numbers, respectively.  Then as an $S$-module, $\MA$
is defined by the following pushout diagram.
\[
\xymatrix@-1pc{%
P_{+}\sma S\ar[r]\ar[d]&P_{+}\sma R\ar[d]\\
\bar P_{+}\sma S\ar[r]&\MA
}
\]
The multiplication on $\MA$ follows the same idea as the multiplication
on the Moore loop space. We think of $r\in P$ as specifying a length,
and we use the action of $\AC(2)$ on $R$ for ``concatenation'':  Given
$r\in P$ and $s\in P$, rescaling the length $r+s$ interval
\[
\underbrace{%
\subseg{r}{10em}%
\subseg{s}{6em}}_{r+s}
\]
specifies an element of $\AC(2)$, with first box length $r/(r+s)$ and
second box length $s/(r+s)$.  We then get a map $P\times P\to
P\times \AC(2)$, sending $(r,s)$ to the length $r+s\in P$ and the
sub-intervals $([0,r/(r+s)],[r/(r+s),1])\in \AC(2)$.  Using the
$\AC$-algebra structure on $R$, we get the \term{concatenation map}
\[
(P_{+}\sma R)\sma_{S} (P_{+}\sma R)\iso (P\times P)_{+}\sma R\sma_{S} R
\to P_{+}\sma \AC(2)_{+}\sma R\sma_{S} R\to P_{+}\sma R.
\]
Since the map $S\to R$ is induced by $\ze\in \AC(0)$, the concatenation
map extends to
a map $\MA\sma \MA\to \MA$.  An easy computation shows that this provides
an associative multiplication on $\MA$, which has unit $S\to \MA$ induced
by the inclusion of $0$ in $\bar P$,
\[
S\iso \{0\}_{+}\sma S\to \bar P_{+}\sma S\to \MA.
\]
We make the following definition.

\begin{defn}\label{defmoore}
The \term{Moore algebra} of $R$ is the associative $S$-algebra
\[
\MA=(\bar P_{+}\sma S)\cup_{(P_{+}\sma S)}(P_{+}\sma R)
\]
with
multiplication induced by the concatenation map as above.
\end{defn}

Dropping the lengths, we obtain a natural map of $S$-modules $\chi
\colon \MA\to R$.  In general the map is not a map of $\AC$-algebras,
but it is a map of associative $S$-algebras if the $\AC$-algebra
structure on $R$ comes from an associative $S$-algebra structure.  We
also have the following analogue of Theorem~\ref{thmACnice}.

\begin{prop}\label{propMAhe}
The map $\chi \colon \MA\to R$ is a homotopy equivalence of $S$-modules.
\end{prop}

To compare the Moore algebra $\MA$ with the left module enveloping
algebra, we construct an algebra $C=CR$ in between.  Let $E=P\times
P\times \bar P$, $\bar E=\bar P\times P\times \bar P$, and define $C$ by
the following pushout diagram of $S$-modules.
\[
\xymatrix@-1pc{%
E_{+}\sma S\ar[r]\ar[d]&E_{+}\sma R\ar[d]\\
\bar E_{+}\sma S\ar[r]&C
}
\]
The multiplication on $C$ combines the multiplications on $\MA$ and
$UR$.  We think of an element $(\ell_{1},\ell_{2},\ell_{3})$ of $E$ as
specifying an interval of length $\ell_{1}+\ell_{2}+\ell_{3}$ together
with sub-intervals of length $\ell_{1}$, $\ell_{2}$, and $\ell_{3}$ in
that order.
\[
\subseg{\ell_{1}}{9em}%
\subseg{\ell_{2}}{2em}%
\subseg{\ell_{3}}{5em}
\]
Given an element $(m_{1},m_{2},m_{3})$ of $E$, the composition
$\circ_{2}$ in $\AC(2)$ has an analogue that associates to
$(\ell_{1},\ell_{2},\ell_{3})$ and $(m_{1},m_{2},m_{3})$ an interval of length
$\ell_{1}+\ell_{2}(m_{1}+m_{2}+m_{3})+\ell_{3}$ with four
sub-intervals of lengths $\ell_{1}$, $\ell_{2}m_{1}$, $\ell_{2}m_{2}$,
and $\ell_{2}m_{3}+\ell_{3}$, in that order.
\[
\subseg{\ell_{1}}{9em}%
\subseg{\ell_{2}m_{1}}{4em}%
\subseg{\ell_{2}m_{2}}{3em}%
\subsegnoend{\ell_{2}m_{3}}{2em}%
\subsegnobeg{\ell_{3}}{5em}
\]
We define a map $E\times E\to E\times \AC(2)$ using the map $E\times
E\to E$ that concatenates the first two sub-intervals, sending
$(\ell_{1},\ell_{2},\ell_{3})$ and $(m_{1},m_{2},m_{3})$ to
\[
(\ell_{1}+\ell_{2}m_{1},\ell_{2}m_{2}, \ell_{2}m_{3}+\ell_{3}),
\]
and using the map $E\times E\to \AC(2)$ that rescales the union of the
first two sub-intervals to length $1$,
\[
\underbrace{%
\subseg{\ell_{1}}{9em}%
\subseg{\ell_{2}m_{1}}{4em}%
}_{\ell_{1}+\ell_{2}m_{1}}
\]
sending $(\ell_{1},\ell_{2},\ell_{3})$ and $(m_{1},m_{2},m_{3})$ to
$([0,\ell_{1}/(\ell_{1}+\ell_{2}m_{1})],[\ell_{1}/(\ell_{1}+\ell_{2}m_{1}),1])$.
The map
\[
(E_{+}\sma R)\sma_{S} (E_{+}\sma R)\iso (E\times E)_{+}\sma R\sma_{S}R
\to E_{+}\sma \AC(2)_{+}\sma R\sma_{S} R\to E_{+}\sma R
\]
extends to a map $C\sma_{S}C\to C$ that provides the multiplication in an
associative $S$-algebra structure.  The unit is induced by the
inclusion of $(0,1,0)$ in $E$,
\[
S\iso \{(0,1,0)\}_{+}\sma S\to \bar E_{+}\sma S\to C.
\]

We can now use $C$ to compare $\MA$ and $UR$ in the category
of associative $S$-algebras.
The embedding $P\to E$ sending $r$ to $(r,1,0)$ and the embedding
$D\to E$ sending $([0,a],[a,b])$ to $(a,b-a,1-b)$ make the following
diagram commute
\[
\xymatrix@R-1pc{%
P\times P\ar[r]\ar[d]&E\times E\ar[d]&D\times D\ar[l]\ar[d]\\
P\times \AC(2)\ar[r]&E\times \AC(2)&D\times \AC(2)\ar[l]
}
\]
and induce maps of associative $S$-algebras
\[
\MA\to CR \from UR.
\]
Looking at Theorem~\ref{thmACnice} and Proposition~\ref{propMAhe} (and
the inverse homotopy equivalences), we
see that these maps are homotopy equivalences of the underlying
$S$-modules.  Thus, we have proved the following theorem.

\begin{thm}\label{thmABC}
The maps of $S$-algebras $\MA\to CR$ and $UR\to CR$ are weak
equivalences and homotopy equivalences of the
underlying $S$-modules.  
\end{thm}

Finally, we explain the relationship of $\MA$ to $R$ in the category
of $\AC$-algebras.  We can choose a zigzag of weak equivalences
\[
R\from R'\to R''
\]
where $R''$ is an associative $S$-algebra. Then in the diagram of
weak equivalences
\[
\xymatrix@-1pc{%
R_{M}&R'_{M}\ar[l]\ar[r]&R''_{M}\ar[d]^{\chi}\\
R&R'\ar[l]\ar[r]&R'',
}
\]
the right vertical arrow is a map of associative $S$-algebras.  This
diagram then gives a zigzag of weak equivalences in the category of
$\AC$-algebras between $R$ and $\MA$.


\bibliographystyle{plain}
\def\noopsort#1{}\def\MR#1{}

\end{document}